\documentclass[11pt, twoside]{article}
\usepackage{fancyhdr}
\usepackage{amsthm}
\newtheorem{theorem}{Theorem}[section]

\newtheorem{lemma}[theorem]{Lemma}
\newtheorem{corollary}[theorem]{Corollary}

\newtheorem{problem}[theorem]{Problem}

\theoremstyle{remark}

\usepackage{cite}
\usepackage[usenames]{color}
\usepackage{titletoc}
\usepackage{latexsym}
\usepackage{amsmath}
\usepackage{amssymb}
\usepackage{multicol}
\usepackage{graphics}
\usepackage{graphicx}
\usepackage{subfigure}
\usepackage{indentfirst}
\usepackage{epsfig}
\usepackage{mathrsfs}
\usepackage[top=2.54cm, bottom=2.54cm, left=2.80cm, right=2.80cm]{geometry}

\usepackage[T1]{fontenc}
\usepackage[utf8]{inputenc}
\usepackage{lmodern}
\usepackage{mathtools}
\usepackage{microtype}
\usepackage{enumitem}
\usepackage{needspace}
\usepackage[hidelinks]{hyperref}
\hypersetup{
  pdftitle={Spectral Turan bounds for bipartite expansions in linear hypergraphs},
  pdfsubject={Spectral Turan bounds for bipartite expansions in linear hypergraphs},
  pdfkeywords={linear hypergraph, bipartite expansion, spectral Turan problem, Berge hypergraph}
}
\allowdisplaybreaks[2]
\newcommand{\F}{\mathbb F}

\newcommand{\one}{\mathbf 1}
\newcommand{\Berge}{\mathcal B}
\newcommand{\exlin}{\operatorname{ex}^{\mathrm{lin}}}
\newcommand{\spexlin}{\operatorname{spex}^{\mathrm{lin}}}
\newcommand{\inj}{\operatorname{inj}}
\newcommand{\Nm}{\operatorname{Nm}}

\newcommand{\norm}[1]{\left\lVert#1\right\rVert}
\newcommand{\doi}[1]{\href{https://doi.org/#1}{\nolinkurl{#1}}}

\title{Spectral Tur\'an problems for complete bipartite expansions in linear hypergraphs}
\author{Lusheng Fang$^{a}$, \, Guorong Gao$^{b,c}$\thanks{This research is supported by the National Key R\&D Program of China (Grant No.~2023YFA1010202), the National Natural Science Foundation of China (Grant No.~12401448), and the Natural Science Foundation of Fujian Province (Grant No.~2024J08030). Email addresses: lusheng\_fang@126.com, grgao@fzu.edu.cn}\\
\small $^a$School of Mathematics and Statistics, Minnan Normal University,\\
\small Zhangzhou, Fujian, China\\
\small $^b$Center for Discrete Mathematics and Theoretical Computer Science,\\
\small Fuzhou University, Fuzhou, Fujian, China\\
\small $^c$School of Mathematics and Statistics, Fuzhou University,\\
\small Fuzhou, Fujian, China}
\date{}
\makeatletter
\renewcommand{\maketitle}{%
  \begin{center}
    {\LARGE\bfseries\@title\par}%
    \ifx\@author\@empty\else
      \vspace{1em}{\large\begin{tabular}[t]{c}\@author\end{tabular}\par}%
    \fi
  \end{center}
  \vspace{0.5em}
  \thispagestyle{plain}%
  \@thanks % Print the funding/e-mail footnote generated by \thanks.
}
\makeatother

\begin{document}
\maketitle
\begin{abstract}
An $r$-uniform hypergraph is linear if any two edges share at most one vertex.
For a family $\mathcal F$ of $r$-graphs, let $\exlin_r(n,\mathcal F)$ and
$\spexlin_r(n,\mathcal F)$ denote, respectively, the maximum number of
edges and the maximum spectral radius of an $n$-vertex $\mathcal F$-free
linear $r$-graph.
In this paper, we show that for $t\ge s\ge3$,
$$ \spexlin_r(n,K_{s,t}^{+})\le \frac{(t-s+1)^{1/s}}{r-1}n^{1-1/s}+o(n^{1-1/s}),$$
where $K_{s,t}^{+}$ is the $r$-graph obtained from the complete bipartite graph $K_{s,t}$ by enlarging each edge with $r-2$ new vertices.
This answers a problem of She, Fan and Kang in a stronger form. Moreover, our result implies that
$$ \operatorname{ex}^{\operatorname{lin}}_r(n,K_{s,t}^{+})
 \le \frac{(t-s+1)^{1/s}}{r(r-1)}n^{2-1/s}
       +o(n^{2-1/s}),$$
which improves the leading coefficient in the corresponding bound of Gao and Chang. We also give a construction which implies that $\spexlin_r(n,K_{s,t}^{+})=\Theta(n^{1-1/s})$ and $\operatorname{ex}^{\operatorname{lin}}_r(n,K_{s,t}^{+})=\Theta(n^{2-1/s})$ for $s\ge 4r-6$ and  $t>(s-1)!$.

% -------
% She, Fan and Kang posed the question of whether, for \(t\ge s\ge2\), 
%  $$\spexlin_r(n,\mathcal B_r(K_{s,t}))\le \frac{(t-s+1)^{1/s}}{r-1}n^{1-1/s}
%  +o\bigl(n^{1-1/s}\bigr).$$
% In this paper, we answer this  problem for $t\ge s\ge3$ by proving the stronger result that
% $$ \spexlin_r(n,K_{s,t}^{+})\le \frac{(t-s+1)^{1/s}}{r-1}n^{1-1/s}+o(n^{1-1/s}).$$
% Since $K_{s,t}^{+}\in\mathcal B_r(K_{s,t})$, our result, together with the previously known case $s=2$,  completely resolves the problem of She, Fan and Kang. Moreover, using the result of Cooper and Dutle that the spectral radius of an $r$-graph is at least its average degree, we obtain
% $$ \operatorname{ex}^{\operatorname{lin}}_r(n,K_{s,t}^{+})
%  \le \frac{(t-s+1)^{1/s}}{r(r-1)}n^{2-1/s}
%        +o(n^{2-1/s}),$$
% which improves the leading coefficient \((t-1)^{1/s}\) in the corresponding bound of Gao and Chang to \((t-s+1)^{1/s}\).
\end{abstract}

\section{Introduction and main results}\label{sec:intro}
An $r$-uniform hypergraph (or $r$-graph for short) $H = (V(H),E(H))$ consists of a vertex set $V(H)$ and an edge set $E(H)$, where $E(H)$ is a set of $r$-element subsets of $V(H)$. As usual, a graph is a $2$-graph. Throughout this paper, we assume that $r\geq 2$. 
The hypergraph $H$ is \emph{linear} if any two distinct edges intersect
in at most one vertex. 

For an $r$-graph $H$ and a family of $r$-graphs $\mathcal{F}$, we say $H$ is $\mathcal{F}$-free if $H$ does not contain any member of $\mathcal{F}$ as a subgraph.
For a family $\mathcal F$ of $r$-graphs, let $\exlin_r(n,\mathcal F)$ and
$\spexlin_r(n,\mathcal F)$ denote, respectively, the maximum number of
edges and the maximum spectral radius of an $n$-vertex $\mathcal F$-free
linear $r$-graph.
Here the spectral radius of an $r$-graph is that of its
adjacency tensor, as defined by Cooper and Dutle~\cite{CooperDutle}. For simple graphs, the subscript $r$ and superscript $\mathrm{lin}$ are omitted from the notation above.

The complete bipartite graph $K_{s,t}$ is one of the fundamental objects in extremal graph theory. The classical theorem of K\H{o}v\'ari--S\'os--Tur\'an states that, for $t\ge s\ge2$,
$$\text{ex}(n,K_{s,t})
 \le \frac{(t-1)^{1/s}}{2}n^{2-1/s}+\frac{s-1}{2}n.$$
Subsequently, F\"uredi~\cite{Furedi} improved this bound  by showing that 
$$\text{ex}(n,K_{s,t})
 \le \frac{(t-s+1)^{1/s}}{2}n^{2-1/s}+o(n^{2-1/s}).$$
Nikiforov~\cite{Nikiforov} later established the corresponding spectral
analogue:
$$\text{spex}(n,K_{s,t})\le (t-s+1)^{1/s}n^{1-1/s}
 +(s-1)n^{1-2/s}+s-2.$$

Given a graph $F$, a hypergraph $H$ is called a Berge-$F$ if there is a bijection $f : E(F) \rightarrow E(H)$ such that $e \subset f (e)$ for each $e \in E(F )$.
Let $\Berge_r(F)$ denote the family of all $r$-uniform
Berge-$F$ hypergraphs. For an integer $r\ge2$, the
\emph{$r$-expansion} of $F$, denoted by $F^{+}$, is the $r$-graph obtained
from $F$ by enlarging every edge of $F$ with $r-2$ new vertices disjoint
from $V(F)$, where distinct edges are enlarged by pairwise disjoint sets of
new vertices.
Clearly, $F^{+}\in\Berge_r(F)$.

Gao and Chang~\cite{GaoChang} proved the following extension of K\H{o}v\'ari--S\'os--Tur\'an's theorem:
 $$\exlin_r(n,K_{s,t}^{+})
 \le \frac{(t-1)^{1/s}}{r(r-1)}n^{2-1/s}+O\bigl(n^{2-2/s}\bigr),$$
where $t\ge s\ge2$.
The spectral analogue was considered by Fang, Gao, Chang and
Hou~\cite{FGCH}. For $t\ge s\ge3$, they showed that
 $$\spexlin_r(n,K_{s,t}^{+})\le\frac{(s^2r^2t-1)^{1/s}}{r-1}n^{1-1/s}+o\bigl(n^{1-1/s}\bigr).$$
In particular, for $t\ge s=2$, they obtained the sharper bound
\begin{equation}\label{SK2t}
\begin{aligned}
\spexlin_r(n,K_{2,t}^{+})\le \frac{\sqrt{t-1}}{r-1}\sqrt n+O(1).
\end{aligned}
\end{equation}
She, Fan and Kang~\cite{SFK} studied the corresponding Berge problem.
For $t\ge s\ge 3$, they proved that 
$$ \spexlin_r\Bigl(n,\{\Berge_r(K_{s,t}),\Berge_r(C_3)\}\Bigl)\le
 \frac{(t-s+1)^{1/s}}{r-1}n^{1-1/s}
 +\frac{s-1}{r-1}n^{1-2/s}+s-2.$$
They further asked whether the  $\Berge_r(C_3)$-free condition can be removed while preserving the same leading term.

\begin{problem}[\cite{SFK}]\label{prob:SFK}
For $t\ge s\ge 2$, does the following inequality hold?
 $$\spexlin_r\bigl(n,\Berge_r(K_{s,t})\bigr)\le \frac{(t-s+1)^{1/s}}{r-1}n^{1-1/s}+o\bigl(n^{1-1/s}\bigr).$$
\end{problem}

The problem has been resolved for $r=2$ by Nikiforov~\cite{Nikiforov}, and the case
$r\ge3$ and $s=2$ follows from~\eqref{SK2t}, since
$K_{2,t}^{+}\in\Berge_r(K_{2,t})$. Thus only the case
$r\ge3$ and $t\ge s\ge3$ remains. We prove the following stronger result, which resolves the remaining case
and may be viewed as a hypergraph analogue of the spectral result of Nikiforov~\cite{Nikiforov}.

\begin{theorem}\label{thm:upper}
Let $r\ge3$ and $t\ge s\ge3$ be integers. Then 
\begin{equation*}
 \spexlin_r\bigl(n,K_{s,t}^+\bigr)
 \le \frac{(t-s+1)^{1/s}}{r-1}n^{1-1/s}
       +O\bigl(n^{1-1/(s-1)}\bigr)
       +O\bigl(n^{1-2/s+2/s^2}\bigr).
\end{equation*}
\end{theorem}

Since every $\Berge_r(K_{s,t})$-free $r$-graph is
$K_{s,t}^{+}$-free, Theorem~\ref{thm:upper}, together with the known
cases $r=2$ and $s=2$, completely resolves Problem~\ref{prob:SFK}.

By the inequality $\rho(H)\ge re(H)/n$ of Cooper
and Dutle~\cite{CooperDutle},
Theorem~\ref{thm:upper} yields the following
corollary, which improves the result of Gao and Chang~\cite{GaoChang} and may be viewed as a hypergraph analogue of the result of F\"uredi~\cite{Furedi}.

\begin{corollary}\label{cor:edges}
Let $r\ge3$ and $t\ge s\ge3$ be integers. Then
\begin{equation*}
 \exlin_r(n,K_{s,t}^{+})
 \le \frac{(t-s+1)^{1/s}}{r(r-1)}n^{2-1/s}
       +O\bigl(n^{2-1/(s-1)}\bigr)
       +O\bigl(n^{2-2/s+2/s^2}\bigr).
\end{equation*}
\end{corollary}

Furthermore, for $r\ge3$, $s\ge4r-6$ and $t>(s-1)!$, we construct $\Berge_r(K_{s,t})$-free linear $r$-graphs with $\Omega(n^{2-1/s})$ edges, extending the lower-bound result of Gao and Chang~\cite{GaoChang}.

\begin{theorem}\label{thm:general-uniformity}
Let $r\ge3$, $s\ge4r-6$ and
$t>(s-1)!$ be integers. Then, for sufficiently large $n$, there is a
$\Berge_r(K_{s,t})$-free linear $r$-graph $H$ on $n$ vertices satisfying
\begin{equation*}
 e(H)\ge\left(\frac{1}{r(r-1)}-o(1)\right)n^{2-1/s}.
 \label{eq:lower-edge}
\end{equation*}
\end{theorem}

Combining the construction in
Theorem~\ref{thm:general-uniformity}
with the upper bounds above, we obtain
the following corollary.

\begin{corollary}\label{cor:order}
Let $r\ge3$, $s\ge4r-6$ and $t>(s-1)!$
be integers. Then
\[
\spexlin_r(n,K_{s,t}^{+})
=\Theta\bigl(n^{1-1/s}\bigr)
\]
and
\[
\exlin_r(n,K_{s,t}^{+})
=\Theta\bigl(n^{2-1/s}\bigr).
\]
\end{corollary}

The rest of this paper is organized as follows.
Section~\ref{sec:prelim} contains notation and preliminary results.
We prove Theorems~\ref{thm:upper} and
\ref{thm:general-uniformity} in Sections~\ref{secproof}
and~\ref{sec:general-lower}, respectively.

\section{Preliminaries}\label{sec:prelim}
In this section, we introduce definitions and notation used throughout the paper
and state some preliminary results.

\subsection{Notation}
For an $r$-graph $H$ and a vertex $v\in V(H)$, let
\[
 N_H(v)=\{u\in V(H)\setminus\{v\}:\{u,v\}\subseteq e
 \text{ for some }e\in E(H)\}.
\]
For two vertices $u,v\in V(H)$, put
\[
 N_H(u,v)=\{x\in V(H)\setminus\{u,v\}:\{x,u,v\}\subseteq e
 \text{ for some }e\in E(H)\}.
\]
For $S=\{x_1,\ldots,x_s\}\in\binom{V(H)}s$, define
 $N_1(S)=\bigcap_{i=1}^s N_H(x_i)$ and
 $$N_2(S)=\bigcup_{1\le i<j\le s}N_H(x_i,x_j).$$
If $a,b$ are adjacent vertices of a linear hypergraph, they belong to a
unique edge; we denote this edge by $h_{ab}$.

The degree of a vertex $v$, denoted by $d_H(v)$, is the number of edges
containing $v$. We write $d(v)$ when the hypergraph is clear, and put
$E_v=\{e\in E(H):v\in e\}$. For $X\subseteq V(H)$ and $0\le i\le r$, let
 $E_i(X)=\{e\in E(H):|e\cap X|=i\}$ and $e_i(X)=|E_i(X)|.$
For $v\in X$, let
 $E_i^v(X)=\{e\in E_i(X):v\in e\}$ and $e_i^v(X)=|E_i^v(X)|.$
For an integer $i$ and a set $S$, let $\binom Si$ denote the family of all
$i$-subsets of $S$. 

\subsection{Spectral radius of hypergraphs}\label{spe}

A real tensor of order $r$ and dimension $n$ $\mathcal{T}=(\mathcal{T}_{i_1\cdots i_r})$ consists of $n^{r}$ real entries $\mathcal{T}_{i_1\cdots i_r}$  for all $i_1, i_2, \ldots, i_r\in  [n]$, where $[n]=\{1,2,\ldots,n\}$. Evidently, a vector of dimension $n$ is a tensor of order $1$ and a matrix is a tensor of order $2$. A tensor $\mathcal{T}$ is called symmetric if the value of $\mathcal{T}_{i_1\cdots i_r}$ is invariant under any permutation of the indices $i_1, i_2, \ldots, i_r$.
In 2005, Qi \cite{Qi05} and Lim \cite{Lim05} independently introduced the concepts of tensor eigenvalues and spectra.  
If there exist $\lambda \in \mathbb{C}$ and a nonzero vector $\mathbf{x} \in \mathbb{C}^n$ satisfying
\begin{equation*}
\begin{aligned}
\mathcal{T}\mathbf{x}^{r-1}=\lambda \mathbf{x}^{[r-1]},
\end{aligned}
\end{equation*}
then $\lambda$ is called an eigenvalue of $\mathcal{T}$ and $\mathbf{x}$ is a corresponding eigenvector, where $\mathbf{x}^{[r-1]}=(x_{1}^{r-1}, x_{2}^{r-1}, \ldots,  x_{n}^{r-1})^T\in \mathbb{C}^n\setminus\{0\}$ and $\mathcal{T}\mathbf{x}^{r-1} \in \mathbb{C}^n$ is defined by
\begin{equation*}
\begin{aligned}
(\mathcal{T}\mathbf{x}^{r-1})_i = \sum_{ i_2, \cdots, i_r \in [n]}\mathcal{T}_{ii_2\cdots i_r}x_{i_2}\cdots x_{i_r}.
\end{aligned}
\end{equation*}
The maximal absolute value of the eigenvalues of $\mathcal{T}$ is called the spectral radius of $\mathcal{T}$, denoted by $\rho(\mathcal{T})$.

In 2012, Cooper and Dutle \cite{CooperDutle} defined the adjacency tensor $\mathcal{A}(H)$ of an $r$-graph $H$. The adjacency tensor $\mathcal{A}(H)$ of $H$ is a symmetric tensor of order $r$ and dimension $n$ defined by
\begin{equation*}
\mathcal{A}_{i_1\cdots i_r}=\left\{
\begin{array}{ll}
\frac{1}{(r-1)!}&  \;\ if\ \{i_1,\ldots, i_r\} \in E,\\
0 &\;\ otherwise.\\
\end{array}\right.
\end{equation*}
For an $r$-graph $H$, the spectral radius of $H$ is defined as the spectral radius of the adjacency tensor $\mathcal{A}(H)$, denoted by $\rho(H)$.  In addition, Qi \cite{Qi05} showed that $\rho(H)$ is also given by the following variational formula:
\begin{equation*} 
\begin{aligned}
\rho(H)=\max_{\|x\|_r=1}\mathbf{x}^T\mathcal{A}(H)\mathbf{x}^{r-1}.
\end{aligned}
\end{equation*}

\subsection{Some useful results}

\begin{lemma}[\cite{GaoChang}]\label{lem:extension}
Let $H$ be a linear $r$-graph, where $r\ge3$, and let $t\ge s\ge2$ be
integers. If some $S\in\binom{V(H)}s$ satisfies
\[
 |N_1(S)\setminus N_2(S)|\ge s^2r^2t,
\]
then $H$ contains $K_{s,t}^{+}$.
\end{lemma}

Since $H$ is linear, $|N_2(S)|\le\binom s2(r-2)$. Hence
Lemma~\ref{lem:extension} immediately gives the following consequence.
\begin{corollary}\label{eq:common-bound}
Let $H$ be a $K_{s,t}^{+}$-free linear $r$-graph. Then, for every
$S\in\binom{V(H)}s$,
\[
 |N_1(S)|\le s^2r^2t-1+\binom s2(r-2).
\]
\end{corollary}

\begin{lemma}[\cite{GaoChang}]\label{lem:loss}
 Let $r\ge3$ and $t\ge s\ge2$ be fixed integers. If $H$ is a $K^+_{s,t}$-free linear $r$-graph, then for any $v \in V(H)$, 
 \begin{equation*}
 \bigl|\{e\in E(H):v\notin e,\ |e\cap N_H(v)|\ge2\}\bigr|
 = O(d(v)^{2-\frac{1}{s-1}}).
 \end{equation*}
\end{lemma}

\begin{corollary}\label{CorGao}
 Let $r\ge3$ and $t\ge s\ge2$ be fixed integers. If $H$ is a $K^+_{s,t}$-free linear $r$-graph, then for any $v \in V(H)$, 
 \begin{equation*}
 \bigl|\{e\in E(H):\ |e\cap N_H(v)|\ge2\}\bigr|
 = O(d(v)^{2-\frac{1}{s-1}}).
 \end{equation*}
\end{corollary}

\begin{lemma}[{\cite{FGCH,SFK}}]\label{lem:quadratic}
Let $H$ be a linear $r$-graph with spectral radius $\rho$. If, for every vertex $u \in V(H)$,
 $$\sum_{v \in N_H(u)}d(v)\leq Pd(u)+Q,$$
then
\begin{equation*}
\begin{aligned}
\rho^2-\frac{P}{r-1}\rho-\frac{Q}{r-1}\leq 0,
\end{aligned}
\end{equation*}
where $P$ and $Q$ are parameters independent of the choice of $u$.
\end{lemma}

\section{Proof of  Theorem \ref{thm:upper}}\label{secproof}
An $r$-graph $H$ is \emph{hm-bipartite} if its vertex set has a partition
$V_1\cup V_2$ such that every edge contains exactly one vertex of $V_1$.
The sets $V_1$ and $V_2$ are called the \emph{head part} and the
\emph{mass part}, respectively. If $|V_1|=m$ and $|V_2|=n$, we call
$H$ an $(m,n)$-hm-bipartite $r$-graph, as in~\cite{FGCH,SFK}.

Suppose that $V_1$ is partitioned into blocks, each of size at most $b$.
A subset of $V_1$ is \emph{transversal} if it meets each block at most
once. A transversal $K_{s,t}^{+}$ has its $s$-part in a transversal subset
of $V_1$ and its $t$-part in $V_2$. All its additional vertices then lie
in $V_2$.

\begin{lemma}\label{lem:basic}
Fix integers $r, t,s\ge 2$, and $b,w\ge 1$.
Let $H$ be an $(m,n)$-hm-bipartite linear $r$-graph satisfying the following properties:
\begin{enumerate}
    \item[(i)] there exists a partition $Q_1,Q_2,\ldots,Q_\ell$ of the head part such that $|Q_i|\le b$ for all $i\in[\ell]$;
    \item[(ii)] every $s$-set of the head part has at most $w$ common neighbors in the mass part;
    \item[(iii)] $H$ contains no copy of a transversal $K_{s,t}^+$.
\end{enumerate}
Then
\begin{equation*}
 e(H)\le\frac{(t-1)^{1/s}}{r-1}mn^{1-1/s}+O\left(m^{2/s}n^{1-2/s^2}\right)+O(n).
\end{equation*}
\end{lemma}
\begin{proof}
Let $V_1$ and $V_2$ be the head and mass parts. If a part is empty, the result is immediate. If $m<s$, then each vertex of $V_2$ has degree at
most $m$, and hence $(r-1)e(H)\le mn=O(n)$. We may therefore assume $m\ge s$ and $n\ge1$.

Let $u_v=\max\{d(v)-s+1, 0\}$ for any $v\in V_2$. Then,  
$d(v)\leq (s-1)+u_v$ and 
\begin{equation}\label{uv}
\begin{aligned}
u^s_v\leq s!\binom{d(v)}{s}.
\end{aligned}
\end{equation}
Hence,  combining these two bounds with H\"older's inequality, we obtain
\begin{equation}\label{edgeupper}
\begin{aligned}
(r-1)e(H)&=\sum_{v\in V_2}d(v)\leq (s-1)n+\sum_{v\in V_2}u_v\\
&\leq (s-1)n+n^{1-1/s}\left(\sum_{v\in V_2}u_v^s\right)^{1/s}\\
&\leq (s-1)n+n^{1-1/s}\left(s!\sum_{v\in V_2}\binom{d(v)}{s}\right)^{1/s}
\end{aligned}
\end{equation}
Let 
$$A=\{(v,\{v_1, v_2, \ldots, v_s \}):v \in V_2, \{v_1, \ldots, v_s \}\subset N_H(v) \cap  V_1\}.$$
Then,  by double counting, we have 
\begin{equation}\label{doublecounting}
\begin{aligned}
\sum_{v \in V_2}\binom{d(v)}{s}=|A|=\sum_{S \in \binom{V_1}{s}} |N_1(S)|.
\end{aligned}
\end{equation} 
Combining \eqref{edgeupper} and \eqref{doublecounting}, we obtain
\begin{equation}\label{edge}
\begin{aligned}
(r-1)e(H)\leq (s-1)n+n^{1-1/s}\left(s!\sum_{S \in \binom{V_1}{s}} |N_1(S)|\right)^{1/s}.
\end{aligned}
\end{equation}
To finish the proof, we derive an upper bound for $\sum_{S \in \binom{V_1}{s}} |N_1(S)|$.

An $s$-set $S \in \binom{V_1}{s}$ is \textbf{bad} if there exist distinct vertices $x,y\in N_1(S)$ such that at least one of the following holds:
\begin{enumerate}
    \item[(i)] there exists $a\in S$ such that $h_{ax}=h_{ay}$;
    \item[(ii)] there exist $a,a'\in S$ and $z\in V_2\setminus\{x,y\}$ such that $z\in h_{ax}\cap h_{a'y}$.
\end{enumerate}
Otherwise, we call $S$  \textbf{good}. %If $S \in \binom{V_1}{s}$ is good, it is clear that $h_{uv}$ are pairwise distinct for all $u\in N_1(S)$ and $v \in N_1(S)$
Note that $H$ contains no copy of a transversal $K_{s,t}^+$. If $S\in \binom{V_1}{s}$ is both  transversal and good, then $|N_1(S)|\leq t-1$. Indeed, if $|N_1(S)|\geq t$, choose $T\subseteq N_1(S)$ with $|T|=t$.
Then, for \((a,x)\in S\times T\),
the sets \(h_{ax}\setminus\{a,x\}\) are pairwise disjoint subsets of $V_2\setminus T$. Hence these hyperedges form a transversal copy of \(K_{s,t}^{+}\), a contradiction.
Let $Y_1$ denote the family of all non-transversal members of $\binom{V_1}{s}$, and $Y_2$ the family of all bad members of $\binom{V_1}{s}$.  Then 
\begin{equation}\label{ssum}
\begin{aligned}
\sum_{S \in \binom{V_1}{s}} |N_1(S)|&=\sum_{S\notin Y_1\cup Y_2} |N_1(S)|+\sum_{S\in Y_1\cup Y_2} |N_1(S)|\\
&\leq (t-1)\left(\binom{m}{s}-|Y_1\cup Y_2|\right)+w|Y_1\cup Y_2|\\
&\leq (t-1)\binom{m}{s}+w|Y_1\cup Y_2|
\end{aligned}
\end{equation}

Next we give an upper bound on $|Y_1|$. To count these sets, we first choose two vertices in a common block $Q_i$, where $i\in[\ell]$, and then choose the remaining \(s-2\) vertices. This yields
\begin{equation}\label{Y1}
\begin{aligned}
|Y_1|&\leq \sum_{i=1}^{\ell}\binom{|Q_i|}{2}\binom{m-2}{s-2}=\sum_{i=1}^{\ell}\frac{|Q_i|-1}{2}|Q_i|\binom{m-2}{s-2}\\
&\leq\frac{b-1}{2}\binom{m-2}{s-2}\sum_{i=1}^{\ell}|Q_i|=\frac{b-1}{2}\binom{m-2}{s-2}m\\
&\leq O(m^{s-1}).
\end{aligned}
\end{equation}

For distinct $x,y\in V_2$, let
$C_{xy}=N_H(x)\cap N_H(y)\cap V_1$ and $c_{xy}=|C_{xy}|$. We call a pair $\{x,y\} \in \binom{V_2}{2}$ bad if at least one of the following holds:
\begin{enumerate}
    \item[(i)] there exists  $a\in C_{xy}$ satisfying  $h_{ax}=h_{ay}$;
    \item[(ii)] there exist $a,a'\in C_{xy}$ and $z\in V_2\setminus\{x,y\}$ such that  $z\in h_{ax}\cap h_{a'y}$.
\end{enumerate}
Let $F$ denote the family of all bad pairs in $\binom{V_2}{2}$. By a simple counting argument, the number of pairs satisfying condition {\rm (i)} is at most $\binom{r-1}{2}e(H),$ while the number of pairs satisfying condition {\rm (ii)} is at most $(r-1)^2\sum_{z\in V_2}d^2(z).$
Consequently,
\begin{equation}\label{F}
\begin{aligned}
|F|\le\binom{r-1}{2}e(H)+(r-1)^2\sum_{z\in V_2}d^2(z).
\end{aligned}
\end{equation}
Moreover, it follows that
\begin{equation}\label{d2}
\begin{aligned}
 \sum_{z\in V_2}d(z)^2&\leq\sum_{z\in V_2}(s-1+u_z)^2\leq 2(s-1)^2n+2\sum_{z\in V_2}u_z^2\\
 &\leq 2(s-1)^2n+2n^{1-2/s}\left(\sum_{z\in V_2}u_z^s\right)^{2/s}\\
 &\leq 2(s-1)^2n+2n^{1-2/s}\left(s!w\binom{m}{s}\right)^{2/s}\\
 &=O(m^2n^{1-2/s})+O(n),
 \end{aligned}
\end{equation}
where the second inequality follows from Jensen's inequality, the third inequality follows from  H\"older's inequality and 
the last inequality follows from 
\eqref{uv}, \eqref{doublecounting} and condition~(ii) of the lemma.
Hence, combining \eqref{F} and \eqref{d2}, we get
\begin{equation}\label{Fe}
\begin{aligned}
|F|&\leq \binom{r-1}{2}e(H)+(r-1)^2\sum_{z\in V_2}d^2(z)\\
&=\binom{r-1}{2}\frac{\sum_{z\in V_2}d(z)}{r-1}+(r-1)^2\sum_{z\in V_2}d^2(z)\\
&\leq \left(\frac{r-2}{2}+(r-1)^2\right)\sum_{z\in V_2}d(z)^2=O(m^2n^{1-2/s})+O(n).
 \end{aligned}
\end{equation}

For every bad $S\in \binom{V_1}{s}$, there exist distinct
$x,y\in N_1(S)$ such that $\{x,y\}\in F$ and $c_{xy}\geq s$.
Fix a pair $\{x,y\}\in F$ with $c_{xy}\geq s$.
If $\{x,y\}$ satisfies condition {\rm (i)}, then the number of bad
$s$-sets arising from condition {\rm (i)} is at most
\[
\binom{c_{xy}-1}{s-1}.
\]
For condition {\rm (ii)}, we only need to consider $a\neq a'$, since
$a=a'$ implies condition {\rm (i)} by the linearity of $H$.
For any $a\in C_{xy}$ and $z\in h_{ax}\setminus\{a,x\}$, if there
exists an edge of $H$ containing both $z$ and $y$, then $a'$ is
uniquely determined. Hence, the number of bad $s$-sets arising from
condition {\rm (ii)} with $a\neq a'$ is at most
\[
c_{xy}(r-2)\binom{c_{xy}-2}{s-2}.
\]
Therefore,
\begin{equation}\label{Y2upper}
\begin{aligned}
|Y_2|&\leq \sum_{\substack{\{x,y\}\in F \\ c_{xy}\geq s}}\left[\binom{c_{xy}-1}{s-1}+c_{xy}(r-2)\binom{c_{xy}-2}{s-2}\right]\\
&\leq O\left(\sum_{\substack{\{x,y\}\in F \\ c_{xy}\geq s}}c_{xy}^{s-1}\right)\leq O\left(|F|^{1/s}\left(\sum_{\substack{\{x,y\} \in \binom{V_2}{2} \\ c_{xy}\geq s}}c_{xy}^s\right)^{(s-1)/s}\right).
\end{aligned}
\end{equation}
Let 
$$B=\left\{\big(S,\{x, y \}\big): S \in \binom{V_1}{s}, \{x, y \}\subset N_1(S) \right\}.$$
Then,  by double counting, we have 
\begin{equation}\label{doublecounting2}
\begin{aligned}
\sum_{\{x,y\} \in \binom{V_2}{2}}\binom{c_{xy}}{s}=|B|=\sum_{S \in \binom{V_1}{s}} \binom{|N_1(S)|}{2}
\leq \binom{w}{2}\binom{m}{s}.
\end{aligned}
\end{equation}
Note that for an integer $c\ge s$, 
\begin{equation}\label{fact}
\begin{aligned}
\binom{c}{s}=\prod_{i=0}^{s-1}\frac{c-i}{s-i}\geq\left(\frac{c}{s}\right)^s
\end{aligned}
\end{equation}
Combining \eqref{doublecounting2} and \eqref{fact}, we obtain 
\begin{equation}\label{Cs}
\begin{aligned}
\sum_{\substack{\{x,y\} \in \binom{V_2}{2} \\ c_{xy}\geq s}}c_{xy}^s\leq O(m^s).
\end{aligned}
\end{equation}
Hence, by \eqref{Y2upper} and \eqref{Cs}, we have
\begin{equation}\label{Y2}
\begin{aligned}
|Y_2|&\leq O\bigg(|F|^{1/s}m^{s-1}\bigg).
\end{aligned}
\end{equation}

Now, combining \eqref{edge}, \eqref{ssum}, \eqref{Y1}, \eqref{Fe} and \eqref{Y2}, we get
\begin{equation*}
\begin{aligned}
(r-1)e(H)&\leq(s-1)n+n^{1-1/s}\left((t-1)m^s+O\bigl((1+|F|^{1/s})m^{s-1}\bigl)\right)^{1/s}\\
&\leq (s-1)n+n^{1-1/s}\left((t-1)^{1/s}m+O(1+|F|^{1/s})\right)\\
&\leq (s-1)n+n^{1-1/s}\left((t-1)^{1/s}m+O\Bigl((m^2n^{1-2/s}+n)^{1/s}\Bigl)\right)\\
&\leq(t-1)^{1/s}mn^{1-1/s}+O\bigg(m^{2/s}n^{1/s-2/s^2}+n^{1/s}\bigg)n^{1-1/s}+O(n)\\
&= (t-1)^{1/s}mn^{1-1/s}+O\bigg(m^{2/s}n^{1-2/s^2}\bigg)+O(n),
\end{aligned}
\end{equation*}
where the second inequality follows from the concavity of
$f(x)= x^{1/s}$ for $x>0$, which gives
$(a+b)^{1/s}\le a^{1/s}+\frac{1}{s}a^{1/s-1}b$
for $a>0$ and $b\ge0$, and
the last inequality follows from
$(a+b)^{1/s}\le a^{1/s}+b^{1/s}$ for $a,b\ge0$. Hence,
\begin{equation*}
\begin{aligned}
e(H)&\leq \frac{(t-1)^{1/s}}{r-1}mn^{1-1/s}+O\bigg(m^{2/s}n^{1-2/s^2}\bigg)+O(n),
\end{aligned}
\end{equation*}
completing the proof.
\end{proof}

%\section{A recursive bipartite estimate}\label{sec:recursion}
%The following proof adapts the induction in She, Fan and Kang~\cite[Theorem~3.5]{SFK} to transversal expansions. We delete the mass neighbors of a chosen vertex in order to extend a smaller expansion by that vertex. Lemma~\ref{lem:loss} bounds the resulting loss of edges.

\begin{lemma}\label{lem:recursive-bipartite}
Fix integers $r\geq 3, s, t\ge 2$, and $b,w\ge 1$.
Let $H$ be an $(m,n)$-hm-bipartite linear $r$-graph
satisfying the following properties:
\begin{enumerate}
    \item[(i)]  there exists a partition $Q_1,Q_2,\ldots,Q_\ell$ of the head part such that $|Q_i|\le b$ for all $i\in[\ell]$;
    \item[(ii)] every $s$-set of $V(H)$ has at most $w$ common neighbors;
    \item[(iii)] $H$ contains no copy of a transversal $K_{s,t}^{+}$.
\end{enumerate}
Then, for every integer $0\le k\le t-2$,
\begin{equation*}
\begin{aligned}
e(H)\le {}&
\frac{(t-k-1)^{1/s}}{r-1}mn^{1-1/s}+O\Bigl(
 m^{2-1/(s-1)}\Bigr)
 +O\Bigl(m^{2/s}n^{1-2/s^2+k(s-2)/s^2}\Bigr)+O\Bigl(n^{1+k/s}
\Bigr).
\end{aligned}
\end{equation*}
\end{lemma}

\begin{proof}
Let $V_1$ and $V_2$ be the head and mass parts of $H$, respectively,
and write $d(v)=d_H(v)$. The case $e(H)=0$ is immediate, so assume
$e(H)>0$ and hence $m,n\ge 1$. 
%Throughout the proof, $C$ denotes a positive constant depending only on $r,s,t,b,w,k$.

We proceed by induction on $k$. For $k=0$, the assertion is given by Lemma~\ref{lem:basic}. Now let $k\ge 1$ and assume that the assertion
holds for all smaller values of $k$ and all corresponding
values of $t$.

Fix $u\in V_2$ and let
$V_{1,u}=N_H(u)\cap V_1$ and
$V_{2,u}=V_2\setminus\bigl(N_H(u)\cup\{u\}\bigr).$
Let $H_u=H[V_{1,u}\cup V_{2,u}]$.
Since $H$ is linear and hm-bipartite, $|V_{1,u}|=d(u)$.
Thus $H_u$ is a $\bigl(d(u),|V_{2,u}|\bigl)$-hm-bipartite linear
$r$-hypergraph. The nonempty sets $Q_i\cap V_{1,u}$,
$i\in[\ell]$, form a partition of its head part, each of size
at most $b$. Every $s$-set of $V(H_u)$ has at most
$w$ common neighbors in $H_u$, since $H_u$ is a subhypergraph
of $H$.

We claim that $H_u$ contains no copy of a transversal
$K_{s,t-1}^{+}$.
Indeed, suppose that such a copy has head core vertices
$v_1,\ldots,v_s$. Clearly, $h_{uv_1}, h_{uv_2}, \ldots ,h_{uv_s}$ are distinct edges that intersect pairwise only at $u$ and contain no other head vertices.
All mass vertices of these edges belong to $N_H(u)\cup\{u\}$ and hence are not in $V_{2,u}$. Adding these edges to the given copy yields a transversal $K_{s,t}^{+}$ in $H$, a contradiction.

Since $k\geq 1$ and $t-2\geq k$, we have $t-1\geq 2$. Thus the parameters $t-1$ and $k-1$ are admissible.
Applying the induction hypothesis to $H_u$ with parameters $t-1$ and $k-1$, and using $|V_{2,u}|\le n$,
we obtain
\begin{equation}\label{eq:Hu-upper}
\begin{aligned}
e(H_u)&\le\frac{(t-k-1)^{1/s}}{r-1}d(u)n^{1-1/s}+O\Bigl(d(u)^{2-1/(s-1)}\Bigr)\\
&+O\Bigl(d(u)^{2/s}n^{1-2/s^2+(k-1)(s-2)/s^2}\Bigr)+O\Bigl(n^{1+(k-1)/s}\Bigr).
\end{aligned}
\end{equation}

On the other hand, condition~(ii) gives $|N_1(S)|\le w$ for
every $S\in\binom{V(H)}{s}$. Thus $H$ is $K_{s,w+1}^{+}$-free. In particular, $H$ is also $K_{s,w+s}^{+}$-free. As $w+s\ge s$, by Corollary~\ref{CorGao}, we have
\begin{equation}\label{Hu}
\begin{aligned}
\bigl|\{e\in E(H):|e\cap N_H(u)|\ge 2\}\bigr| \le O\Bigl(d(u)^{2-1/(s-1)}\Bigl).
\end{aligned}
\end{equation}
Let $E_1=\{e\in E(H): e\cap V_{1,u}\neq \emptyset, e\notin E(H_u) \}$. Clearly, for any $e\in E_1$, $|e\cap N_H(u)|\ge 2$.
Then, by \eqref{Hu}, we have
\begin{equation}\label{eq:Hu-lower}
\begin{aligned}
e(H_u)&= \sum_{v\in V_{1,u}}d(v)-|E_1|\\
&\ge \sum_{v\in V_{1,u}}d(v)-O\Bigl(d(u)^{2-1/(s-1)}\Bigl).
\end{aligned}
\end{equation}
Combining \eqref{eq:Hu-upper} and \eqref{eq:Hu-lower}, we obtain
\begin{equation}\label{eq:local-degree-sum}
\begin{aligned}
\sum_{v\in V_{1,u}}d(v)\le &
\frac{(t-k-1)^{1/s}}{r-1}d(u)n^{1-1/s}+O\Bigl(d(u)^{2-1/(s-1)}\Bigr)\\
&+O\Bigl(d(u)^{2/s}n^{1-2/s^2+(k-1)(s-2)/s^2}\Bigr)+O\Bigl(n^{1+(k-1)/s}\Bigr).
\end{aligned}
\end{equation}

For $u\in V_2$, we have $d(u)\le m$ and
$\sum_{u\in V_2}d(u)=(r-1)e(H)$. Hence
\[
\begin{aligned}
\sum_{u\in V_2}d(u)^{2-1/(s-1)}
&\le m^{1-1/(s-1)}\sum_{u\in V_2}d(u) \\
&=(r-1)m^{1-1/(s-1)}e(H).
\end{aligned}
\]
By H\"older's inequality, 
\[
\sum_{u\in V_2}d(u)^{2/s}
\le n^{1-2/s}
\left(\sum_{u\in V_2}d(u)\right)^{2/s}
=n^{1-2/s}\bigl((r-1)e(H)\bigr)^{2/s}.
\]
Summing \eqref{eq:local-degree-sum} over $u\in V_2$ and combining these bounds gives
\begin{equation}\label{eq:sum-upper}
\begin{aligned}
\sum_{u\in V_2}\sum_{v\in V_{1,u}}d(v)
\le {}&
(t-k-1)^{1/s}n^{1-1/s}e(H)
+O\Bigl(m^{1-1/(s-1)}\Bigl)e(H) \\
&+O\Bigl(n^{2-2/s-2/s^2+(k-1)(s-2)/s^2})e(H)^{2/s}
+O\Bigl(n^{2+(k-1)/s}\Bigl).
\end{aligned}
\end{equation}
Moreover, it follows that
\begin{equation}\label{eq:sum-lower}
\begin{aligned}
\sum_{u\in V_2}\sum_{v\in V_{1,u}}d(v)
&=\sum_{v\in V_1}\sum_{u\in N_H(v)\cap V_2}d(v) \\
&=(r-1)\sum_{v\in V_1}d(v)^2 \\
&\ge \frac{r-1}{m}
\left(\sum_{v\in V_1}d(v)\right)^2
=\frac{r-1}{m}e(H)^2.
\end{aligned}
\end{equation}
Combining \eqref{eq:sum-upper} and \eqref{eq:sum-lower}, we get
\begin{equation}\label{eq:almost-quadratic}
\begin{aligned}
e(H)^2\le {}&
\left(
\frac{(t-k-1)^{1/s}}{r-1}mn^{1-1/s}
+O\Bigl(m^{2-1/(s-1)}\Bigl)\right)e(H) \\
&+O\Bigl(m n^{2-2/s-2/s^2+(k-1)(s-2)/s^2}\Bigl)e(H)^{2/s}
+O\Bigl(m n^{2+(k-1)/s}\Bigl).
\end{aligned}
\end{equation}

If $e(H)\le \frac{(t-k-1)^{1/s}}{r-1}mn^{1-1/s},$
then this completes the proof. Now suppose $e(H)> \frac{(t-k-1)^{1/s}}{r-1}mn^{1-1/s}$.
Then
\begin{equation}\label{es2}
\begin{aligned}
e(H)^{2/s}
&=e(H)e(H)^{2/s-1} \\
&\le
\left(
\frac{(t-k-1)^{1/s}}{r-1}mn^{1-1/s}
\right)^{2/s-1}e(H) \\
&\le O\Bigl(m^{2/s-1}n^{(1-1/s)(2/s-1)}\Bigl)e(H).
\end{aligned}
\end{equation}
Combining \eqref{eq:almost-quadratic} and \eqref{es2}, we obtain
\begin{align*}
&e(H)-
\frac{(t-k-1)^{1/s}}{r-1}mn^{1-1/s}\\
&\quad\le
O\bigl(m^{2-1/(s-1)}\bigr)
+O\bigl(m^{2/s}n^{1-2/s^2+k(s-2)/s^2}\bigr)
+O\left(\frac{mn^{2+(k-1)/s}}{e(H)}\right)\\
&\quad\le
O\bigl(m^{2-1/(s-1)}\bigr)
+O\bigl(m^{2/s}n^{1-2/s^2+k(s-2)/s^2}\bigr)
+O\bigl(n^{1+k/s}\bigr),
\end{align*}
where the last inequality follows from 
$e(H)> \frac{(t-k-1)^{1/s}}{r-1}mn^{1-1/s}$. This completes the proof.
\end{proof}

\begin{lemma}\label{prop:neighbor}
Let $r\ge3$ and $t\ge s\ge3$ be fixed integers, and let $H$ be a $K_{s,t}^+$-free linear $r$-graph on $n$ vertices.  Then for any $u \in V(H)$,
\begin{equation}\label{eq:neighbor-linear}
\begin{aligned}
\sum_{v\in N_H(u)}d(v)\le {}&
\left((t-s+1)^{1/s}n^{1-1/s}+O(n^{1-1/(s-1)})+O(n^{1-2/s+2/s^2})\right)d(u)\\
&+O(n^{2-3/s+2/s^2}).
\end{aligned}
\end{equation}
\end{lemma}

\begin{proof}
For any $u \in V(H)$, by counting the edges incident with each vertex of $N_H(u)$, we have
\begin{equation}\label{eq:neighbor-decomposition}
\sum_{v\in N_H(u)}d(v)
=\sum_{v\in N_H(u)}e_1^v(N_H(u))
 +\sum_{v\in N_H(u)}\sum_{i=2}^r e_i^v(N_H(u)).
\end{equation}
Since $H$ is $K_{s,t}^+$-free, Corollary~\ref{CorGao} gives
\begin{equation}\label{eq:outer-loss}
\begin{aligned}
\sum_{v\in N_H(u)}\sum_{i=2}^r e_i^v(N_H(u))
&=\sum_{i=2}^r i e_i(N_H(u))\\
&\le r\bigl|\{e\in E(H):|e\cap N_H(u)|\ge2\}\bigr|\\
&= O(d(u)^{2-1/(s-1)}).
\end{aligned}
\end{equation}

To complete the proof, we next establish an upper bound for $\sum_{v\in N_H(u)}e_1^v(N_H(u))$.
Let $H_1$ be a subhypergraph of $H$ with
$V(H_1)=V(H)\setminus\{u\}$ and $E(H_1)=E_1(N_H(u)).$
Clearly, $H_1$ is an hm-bipartite linear $r$-graph with head part $V_1=N_H(u)$ and mass part $V_2=V(H) \backslash (N_H(u) \cup \{u\})$.
Then, $|V_1|=(r-1)d(u), |V_2|=n-(r-1)d(u)-1,$
and $e(H_1)=\sum_{v\in N_H(u)}e_1^v(N_H(u)).$
Let $E_u=\{e_1,\dots,e_\ell\}$ with $\ell=d(u)$, and let $Q_i=e_i\setminus\{u\}$ for each $i\in[\ell]$. Then the sets $Q_1,\dots,Q_\ell$ form a partition of $V_1$ into $(r-1)$-element subsets.
Furthermore, we claim that $H_1$ contains no transversal $K_{s,t-1}^{+}$.
Suppose for contradiction that $H_1$ contains a transversal copy of  $K_{s,t-1}^{+}$. Let $v_1,\ldots,v_s$ denote the vertices in the $s$-part of this copy. By definition, these vertices belong to distinct blocks of $V_1$. It is clear that $|h_{uv_i}\cap V_1|=r-1$ for any $i\in[s]$
and all vertices of this copy other than $v_1,\ldots,v_s$ belong to $V_2$. 
Adding the edges $h_{uv_1},\ldots,h_{uv_s}$ to the copy gives a $K_{s,t}^{+}$ in $H$,
a contradiction. Hence, by Corollary~\ref{eq:common-bound} and Lemma \ref{lem:recursive-bipartite}, for every  $0\le k\le t-3$, we have
\begin{equation}
\begin{aligned}
e(H_1)\le {}&
\frac{(t-k-2)^{1/s}}{r-1}|V_1||V_2|^{1-1/s}+O\Bigl(|V_1|^{2-1/(s-1)}\Bigl)\\
&+O\Bigl(|V_1|^{2/s}|V_2|^{1-2/s^2+k(s-2)/s^2}\Bigl)
+O\Bigl(|V_2|^{1+k/s}\Bigr).
\end{aligned}
\end{equation}
Taking $k=s-3$, we obtain
\begin{equation}\label{eq:outer-bipartite}
\begin{aligned}
e(H_1)\le {}&
\frac{(t-s+1)^{1/s}}{r-1}|V_1||V_2|^{1-1/s}+O\Bigl(|V_1|^{2-1/(s-1)}\Bigl)\\
&+O\Bigl(|V_1|^{2/s}|V_2|^{2-5/s+4/s^2}\Bigl)
       +O\Bigl(|V_2|^{2-3/s}\Bigr).
\end{aligned}
\end{equation}

Now, combining \eqref{eq:neighbor-decomposition}, \eqref{eq:outer-loss}
and \eqref{eq:outer-bipartite},
we obtain
\begin{equation}\label{eq:outer-raw}
\begin{aligned}
\sum_{v\in N_H(u)}d(v)\le {}&
(t-s+1)^{1/s}d(u)\bigl(n-(r-1)d(u)\bigr)^{1-1/s}\\
&+O\Bigl(d(u)^{2-1/(s-1)}\Bigl)+O\Bigl(n^{2-5/s+4/s^2}d(u)^{2/s}\Bigl)+O\Bigl(n^{2-3/s}\Bigl)\\
&\leq(t-s+1)^{1/s}n^{1-1/s}d(u)+O(n^{1-1/(s-1)})d(u)\\
&+O\bigg(n^{2-5/s+4/s^2}d(u)^{2/s}\bigg)+O(n^{2-3/s}).
\end{aligned}
\end{equation}
The function $f(x)=x^{2/s}$ is concave for $x>0$, so its tangent line at any $q>0$ gives $f(x)\leq f(q)+f'(q)(x-q)$. Taking $x=d(u)$ and $q=n^{1-1/s}$, we obtain
\[
d(u)^{2/s}
\le\frac2s n^{-1+3/s-2/s^2}d(u)
  +\left(1-\frac2s\right)n^{2/s-2/s^2}.
\]
Multiplying both sides by $n^{2-5/s+4/s^2}$, we get
\begin{equation}\label{eq:tangent}
n^{2-5/s+4/s^2}d(u)^{2/s}
\le\frac2s n^{1-2/s+2/s^2}d(u)
 +\left(1-\frac2s\right)n^{2-3/s+2/s^2}.
\end{equation}

Finally, combining \eqref{eq:outer-raw} and
\eqref{eq:tangent}, we have
\begin{equation*}
\begin{aligned}
\sum_{v\in N_H(u)}d(v)\le {}&
\left((t-s+1)^{1/s}n^{1-1/s}+O(n^{1-1/(s-1)})+O(n^{1-2/s+2/s^2})\right)d(u)\\
&+O(n^{2-3/s+2/s^2}),
\end{aligned}
\end{equation*}
which completes the proof.
\end{proof}

\begin{proof}[Proof of Theorem~\ref{thm:upper}]
Let $H$ be a $K_{s,t}^{+}$-free linear $r$-graph on $n$ vertices, and let
$A=(t-s+1)^{1/s}n^{1-1/s}.$
By Lemmas~\ref{lem:quadratic} and \ref{prop:neighbor}, we have
$$\rho(H)^2
\leq
\frac{
A+O\bigl(n^{1-1/(s-1)}\bigr)
+O\bigl(n^{1-2/s+2/s^2}\bigr)
}{r-1}\rho(H)
+O\bigl(n^{2-3/s+2/s^2}\bigr).$$
If $\rho(H)\leq \frac{A}{r-1},$
this completes the proof. Thus, we may assume that
$\rho(H)>\frac{A}{r-1}.$
Dividing the above inequality by $\rho(H)$ and rearranging, we obtain
\[
\begin{aligned}
\rho(H)
\leq &\frac{A}{r-1}+
O\bigl(n^{1-1/(s-1)}\bigr)
+O\bigl(n^{1-2/s+2/s^2}\bigr)+O\left(
\frac{n^{2-3/s+2/s^2}}{\rho(H)}
\right)\\
\leq& \frac{(t-s+1)^{1/s}}{r-1}n^{1-1/s}
+O\bigl(n^{1-1/(s-1)}\bigr)
+O\bigl(n^{1-2/s+2/s^2}\bigr),
\end{aligned}
\]
where the last inequality follows from 
$\rho(H)>
\frac{A}{r-1}
=\frac{(t-s+1)^{1/s}}{r-1}n^{1-1/s}$. This completes the proof.

% Let $P=(t-s+1)^{1/s}n^{1-1/s}+O(n^{1-1/(s-1)})+O(n^{1-2/s+2/s^2})$ and $Q=O(n^{2-3/s+2/s^2})$. Then, 
% by Lemmas~\ref{lem:quadratic} and \ref{prop:neighbor}, it follows that
% \[
%  \rho(H)^2-\frac{P}{r-1}\rho(H)-\frac{Q}{r-1}\le0.
% \]
% Hence,
% \begin{align*}
% \rho(H)
% &\leq
% \frac{P+\sqrt{P^2+4(r-1)Q}}{2(r-1)}\leq
% \frac{P}{r-1}
% +\sqrt{\frac{Q}{r-1}}\\
% &\leq
% \frac{(t-s+1)^{1/s}}{r-1}n^{1-1/s}
% +O\bigl(n^{1-1/(s-1)}\bigr)
% +O\bigl(n^{1-3/(2s)+1/s^2}\bigr),
% \end{align*}
% as desired.
\end{proof}

\section{Proof of  Theorem \ref{thm:general-uniformity}}\label{sec:general-lower}
We first define the norm graph and state the properties
used in the proof.

Let $s\ge3$ be a fixed integer, and let $q$
be an odd prime power. Let
$\Nm:\F_{q^{s-1}}\to\F_q$ be the field norm
defined by
$\Nm(A)=A^{1+q+\cdots+q^{s-2}}.$
The norm graph $G_q$ of Alon, R\'onyai and
Szab\'o~\cite{ARS} has vertex set
$V(G_q)=\F_{q^{s-1}}\times\F_q^*.$ Write $n_q=|V(G_q)|$. Then
$n_q=q^{s-1}(q-1).$
Two distinct vertices $(A,u)$ and $(B,v)$
are adjacent if and only if
$\Nm(A+B)=uv.$

We use the following properties of $G_q$.

\begin{lemma}[\cite{ARS}]
\label{lem:quotient-norm}
The graph $G_q$ is $K_{s,t}$-free for
$t>(s-1)!$. Moreover,
\begin{equation*}
e(G_q)=\left(\frac12+o(1)\right)
n_q q^{s-1}.
\end{equation*}
\end{lemma}

For simple graphs $F$ and $G$, let $\inj(F,G)$
denote the number of injective maps from $V(F)$
to $V(G)$ that preserve edges. We use the
following spectral counting lemma in a matrix
form that allows nonzero diagonal entries~\cite{AlonPudlak,BMRS}.
Here $\one$ denotes the all-ones vector,
$\one^\perp$ its orthogonal complement, and
$\norm{\cdot}$ the Euclidean operator norm.

\begin{lemma}[\cite{AlonPudlak,BMRS}]\label{lem:counting}
Let $\Delta\ge 1$ be an integer, and let $M$
be a symmetric $0$--$1$ matrix of order $n$ with
constant row sum $d>0$ such that
$\norm{M|_{\one^\perp}}\le\lambda$ and
\[
\lambda=o\left(\frac{d^\Delta}{n^{\Delta-1}}\right).
\]
Then the simple graph $G$ obtained by deleting all loops from the graph with adjacency matrix $M$ satisfies
\[
\inj(F,G)=(1+o(1))n^{v(F)}
\left(\frac dn\right)^{e(F)}
\]
for every fixed simple graph $F$ with
$\Delta(F)\le\Delta$.
\end{lemma}

We next recall the spectral bound for
projective norm graphs obtained by
Szab\'o~\cite{SzaboSpectrum};
see also~\cite{BMRS}.

\begin{lemma}[\cite{SzaboSpectrum,BMRS}]
\label{lem:norm-spectrum}
Let $s\ge 3$ be an integer, and let $q$
be an odd prime power. Let
$W=\F_{q^{s-1}}\times\F_q^*,$ and let $M_1$ be the symmetric $0$--$1$ matrix
indexed by $W$ whose entry corresponding
to $(A,u)$ and $(B,v)$ is one if and only if
$\Nm(A+B)=uv,$
with nonzero diagonal entries allowed.
Then $M_1$ has constant row sum
$q^{s-1}-1$ and
$\norm{M_1|_{\one^\perp}}
\le q^{(s-1)/2}.$
\end{lemma}

We now establish a counting result for the
norm graph $G_q$.

\begin{lemma}\label{lem:quotient-counts}
Let $r\ge3$ and $s\ge4r-6$ be  integers, and let $F$ be a  simple graph with
$\Delta(F)\le 2r-3$.
Then, for sufficiently large $q$, we have
\begin{equation*}
\inj(F,G_q)=(1+o(1))n_q^{v(F)}\left(\frac{q^{s-1}-1}{n_q}\right)^{e(F)}.
\end{equation*}
\end{lemma}
\begin{proof}
Let $M$ be the symmetric $0$--$1$ matrix
indexed by $V(G_q)$, defined by
\[
M_{(A,u),(B,v)}=
\begin{cases}
1,& \Nm(A+B)=uv,\\
0,& \text{otherwise},
\end{cases}
\]
where nonzero diagonal entries are allowed.
Thus the adjacency matrix of $G_q$ is
obtained from $M$ by setting its diagonal
entries to zero.
By Lemma~\ref{lem:norm-spectrum},
$M$ has constant row sum $d=q^{s-1}-1$ and
$\norm{M|_{\one^\perp}}\le q^{(s-1)/2}.$

Let $p=(q^{s-1}-1)/n_q.$
Since $n_q=q^{s-1}(q-1)$ and $s\ge4r-6$,
it follows that
\[
\begin{aligned}
\frac{q^{(s-1)/2}}{n_qp^{2r-3}}=O\left(q^{2r-3-(s+1)/2}\right)=O(q^{-1/2})=o(1).
\end{aligned}
\]
Thus, taking $\Delta=2r-3$ and
$\lambda=q^{(s-1)/2}$, we have
\[
\lambda
=o\left(\frac{d^\Delta}{n_q^{\Delta-1}}\right).
\]
Applying Lemma~\ref{lem:counting}
completes the proof.
\end{proof}

We use the following consequence of the
Pippenger--Spencer theorem.

\begin{theorem}[\cite{PippengerSpencer}]
\label{thm:coloring}
Let $r\ge 3$ be an  integer. For sufficiently large $n$ and $d>0$, if $\mathcal H$
is an $r$-graph on $n$ vertices such that
$\Delta(\mathcal H)\le (1+o(1))d$ and
$\Delta_2(\mathcal H)=o(d)$,
then $\mathcal H$ has a proper edge-coloring with at most
$(1+o(1))d$ colors. In particular, if
\[
e(\mathcal H)=(1+o(1))\frac{nd}{r},
\]
then $\mathcal H$ contains a matching of size
$(1-o(1))n/r$.
\end{theorem}

The following lemma gives a matching criterion in terms of the
second moments of degrees and codegrees.

\begin{lemma}\label{lem:moment-matching}
Let $r\ge 3$ be a fixed integer. For sufficiently large
$n$, let $d>0$ and let $\mathcal H$ be an $r$-graph on
$n$ vertices satisfying
\begin{enumerate}
    \item[(i)]  $e(\mathcal H)=(1+o(1))\frac{nd}{r}$;
    \item[(ii)] $\sum_{x\in V(\mathcal H)}(d_{\mathcal H}(x)-d)^2
=o(nd^2)$;
    \item[(iii)] $\sum_{\{x,y\}\in\binom{V(\mathcal H)}{2}}d_{\mathcal H}(x,y)^2=o(nd^2).$
\end{enumerate}
Then $\mathcal H$ contains a matching of size
$(1-o(1))n/r$.
\end{lemma}

\begin{proof}
For convenience, let $S_1=\sum_{x\in V(\mathcal H)}(d_{\mathcal H}(x)-d)^2$ and $S_2=\sum_{\{x,y\}\in\binom{V(\mathcal H)}{2}}
d_{\mathcal H}(x,y)^2.$ We may then choose
$\varepsilon>0$ with $\varepsilon=o(1)$ such that $S_1
=o(\varepsilon^2 nd^2)$ and $S_2 =o(\varepsilon^2nd^2).$
Let
$B=\{x\in V(\mathcal H): d_{\mathcal H}(x)>(1+\varepsilon)d\}.$
It follows that
\[
|B|\le \frac{S_1}{\varepsilon^2d^2}=o(n).
\]
Thus, by the Cauchy--Schwarz inequality,
\[
\begin{aligned}
\sum_{x\in B}d_{\mathcal H}(x)&=\sum_{x\in B}(d+d_{\mathcal H}(x)-d)
= d|B|+\sum_{x\in B}(d_{\mathcal H}(x)-d)\\
&\leq d|B|+\sqrt{|B|S_1}=o(nd).
\end{aligned}
\]
Moreover,
\[
\begin{aligned}
&\sum_{\substack{\{x,y\}\in\binom{V(\mathcal H)}2\\
d_{\mathcal H}(x,y)>\varepsilon d}}
d_{\mathcal H}(x,y)\le \frac{S_2}{\varepsilon d}
=o(nd).
\end{aligned}
\]

Let $\mathcal H'$ be obtained from $\mathcal H$
by deleting all edges meeting $B$ and all edges
containing a pair $\{x,y\}$ with
$d_{\mathcal H}(x,y)>\varepsilon d$.
Then
$e(\mathcal H')=(1+o(1))\frac{nd}{r},$
$\Delta(\mathcal H')\le(1+\varepsilon)d$ and
$\Delta_2(\mathcal H')\le\varepsilon d.$
By Theorem~\ref{thm:coloring}, $\mathcal H'$
has a proper edge-coloring with at most
$(1+o(1))d$ colors.
Since each color class is a matching,
$\mathcal H'$ contains a matching of size at least
\[
\frac{e(\mathcal H')}{(1+o(1))d}
=(1-o(1))\frac nr.
\]
This completes the proof.
\end{proof}

\begin{proof}[Proof of Theorem~\ref{thm:general-uniformity}]
By the prime number theorem, we may choose
an odd prime $q\le n^{1/s}$ such that
$q=(1-o(1))n^{1/s}$.
Let $G_q$ be the $K_{s,t}$-free norm
graph of Alon, R\'onyai and Szab\'o~\cite{ARS}.
Then,
$n_q=|V(G_q)|=q^{s-1}(q-1).$
Let $p=(q^{s-1}-1)/n_q$. By
Lemma~\ref{lem:quotient-counts}, the number of copies
of $K_r$ in $G_q$ is
\[
\frac{\inj(K_r,G_q)}{r!}
=(1+o(1))\frac{n_q^rp^{\binom r2}}{r!}.
\]

We now construct an auxiliary hypergraph
$\mathcal H$ with
$V(\mathcal H)=E(G_q)$ and
\[
E(\mathcal H)=
\{E(K):K\text{ is a copy of }K_r
\text{ in }G_q\}.
\]
Clearly, $\mathcal H$ is a $\binom r2$-graph,
and matchings in $\mathcal H$ correspond
to families of copies of $K_r$ in $G_q$ with pairwise
disjoint edge sets.
Let $r'=\binom r2, n'=e(G_q)$ and $d=(n_q^{r-2}p^{r'-1})/(r-2)!.$ Then, by Lemma~\ref{lem:quotient-norm},
$n'=(1+o(1))\frac{n_q^2p}{2}.$
Since $n_q=\Theta(q^s)$ and
$p=\Theta(q^{-1})$, we have
$d=\Theta\left(
q^{s(r-2)-r'+1}
\right)$. It follows that
\begin{equation}\label{eq:clique-first}
\begin{aligned}
e(\mathcal H)=\frac{\inj(K_r,G_q)}{r!}=(1+o(1))\frac{n_q^rp^{r'}}{r!}=(1+o(1))\frac{n'd}{r'}.
\end{aligned}
\end{equation}

For $2\le\ell\le r$, let $Q_\ell$ be the number
of ordered pairs $(K,L)$ of copies of $K_r$
in $G_q$ with
$|V(K)\cap V(L)|=\ell$.
We allow $K=L$ when $\ell=r$.
Let $F_\ell$ be the union of two copies of $K_r$ sharing exactly $\ell$ vertices. It is clear that $v(F_\ell)=2r-\ell, e(F_\ell)=2\binom r2-\binom\ell2$ and $\Delta(F_\ell)=2r-\ell-1\le2r-3.$
Since each ordered pair counted by $Q_\ell$
corresponds to exactly
$\ell!((r-\ell)!)^2$ embeddings of $F_\ell$, Lemma~\ref{lem:quotient-counts} gives
\begin{equation*}
\begin{aligned}
Q_\ell=\frac{\inj(F_\ell,G_q)}
{\ell!((r-\ell)!)^2}=(1+o(1))
\frac{n_q^{2r-\ell}
p^{2\binom r2-\binom\ell2}}
{\ell!((r-\ell)!)^2}.
\end{aligned}
\end{equation*}
Thus, we obtain
\begin{equation*}
\begin{aligned}
Q_2=(1+o(1))
\frac{n_q^{2r-2}p^{2r'-1}}
{2((r-2)!)^2}=(1+o(1))n'd^2.
\end{aligned}
\end{equation*}
For $3\le\ell\le r$, the assumption $s\ge4r-6$ gives
\begin{equation*}
\begin{aligned}
\frac{Q_\ell}{n'd^2}
&=O\left(
n_q^{2-\ell}p^{1-\binom\ell2}
\right)=O\left(
q^{-(\ell-2)(s-(\ell+1)/2)}
\right)
=o(1).
\end{aligned}
\end{equation*}
Double counting and the preceding two estimates give
\[
\begin{aligned}
\sum_{x\in V(\mathcal H)}d_{\mathcal H}(x)^2=\sum_{\ell=2}^r\binom\ell2 Q_\ell=(1+o(1))n'd^2.
\end{aligned}
\]
On the other hand, by~\eqref{eq:clique-first},
\[
\sum_{x\in V(\mathcal H)}d_{\mathcal H}(x)
=r'e(\mathcal H)=(1+o(1))n'd.
\]
It follows that
\begin{equation}\label{Con2}
\begin{aligned}
&\sum_{x\in V(\mathcal H)}
(d_{\mathcal H}(x)-d)^2=
\sum_{x\in V(\mathcal H)}d_{\mathcal H}(x)^2
-2d\sum_{x\in V(\mathcal H)}d_{\mathcal H}(x)
+n'd^2=o(n'd^2).
\end{aligned}
\end{equation}
Similarly, double counting gives
\begin{equation}\label{Con3}
\begin{aligned}
&\sum_{\{x,y\}\in\binom{V(\mathcal H)}2}
d_{\mathcal H}(x,y)^2=
\sum_{\ell=2}^r
\binom{\binom\ell2}{2}Q_\ell=o(n'd^2).
\end{aligned}
\end{equation}

By~\eqref{eq:clique-first}, \eqref{Con2},
\eqref{Con3}, and Lemma~\ref{lem:moment-matching},
$G_q$ contains a family $\mathcal K$ of copies
of $K_r$ with pairwise disjoint edge sets
such that
\[
|\mathcal K|
=(1-o(1))\frac{e(G_q)}{\binom r2}.
\]
Let $H_q$ be an $r$-graph on vertex set $V(G_q)$ with edge set
$E(H_q)=\{V(K):K\in\mathcal K\}.$
Then, by Lemma~\ref{lem:quotient-norm} and
$n_q=(1+o(1))q^s$,
\[
\begin{aligned}
e(H_q)=(1-o(1))
\frac{e(G_q)}{\binom r2}=\left(\frac1{r(r-1)}+o(1)\right)
n_qq^{s-1}=\left(
\frac{1}{r(r-1)}+o(1)
\right)n_q^{2-1/s}.
\end{aligned}
\]
Since any two copies in $\mathcal K$ share
at most one vertex, $H_q$ is linear.
Since $G_q$ is $K_{s,t}$-free, $H_q$ is $\Berge_r(K_{s,t})$-free.

Finally, adding $n-n_q$ isolated vertices
to $H_q$, we obtain a
$\Berge_r(K_{s,t})$-free linear $r$-graph
$H$ on $n$ vertices.
Since $n_q=(1-o(1))n$, we have
\[
\begin{aligned}
e(H)=e(H_q)
&\ge
\left(\frac1{r(r-1)}-o(1)\right)
n^{2-1/s}.
\end{aligned}
\]
This completes the proof.
\end{proof}

\section*{Declaration on the use of generative AI}

The authors used generative AI tools to assist in discussing proof strategies, checking proofs, and
improving the exposition.
This project has been under investigation since 2023. 
The authors take full responsibility for the mathematical arguments, results,
and conclusions, all of which were carefully reviewed and verified by them.


\begin{thebibliography}{99}

\bibitem{AlonPudlak}
N.~Alon, P.~Pudl\'ak,
Constructive lower bounds for off-diagonal Ramsey numbers,
\emph{Israel J. Math.}, 122: 243--251, 2001.

\bibitem{ARS}
N.~Alon, L.~R\'onyai, T.~Szab\'o,
Norm-graphs: variations and applications,
\emph{J. Combin. Theory Ser. B}, 76: 280--290, 1999.

\bibitem{BMRS}
T.~Bayer, T.~M\'esz\'aros, L.~R\'onyai, T.~Szab\'o,
Subgraphs of the projective norm graph,
\emph{Forum Math. Sigma}, 14: e121, 2026.

\bibitem{CooperDutle}
J.~Cooper, A.~Dutle,
Spectra of uniform hypergraphs,
\emph{Linear Algebra Appl.}, 436: 3268--3292, 2012.

\bibitem{FGCH}
L.~Fang, G.~Gao, A.~Chang, Y.~Hou,
On the spectral Tur\'an problems for bipartite hypergraphs,
\emph{Electron. J. Combin.}, 32(4): \#P4.53, 2025.

\bibitem{Furedi}
Z.~F\"uredi,
An upper bound on Zarankiewicz' problem,
\emph{Combin. Probab. Comput.}, 5: 29--33, 1996.

\bibitem{GaoChang}
G.~Gao, A.~Chang,
A linear hypergraph extension of the bipartite Tur\'an problem,
\emph{European J. Combin.}, 93: 103269, 2021.

\bibitem{Lim05}
L.-H.~Lim,
Singular values and eigenvalues of tensors: a variational approach,
in: \emph{Proceedings of the IEEE International Workshop on
Computational Advances in Multi-Sensor Adaptive Processing
(CAMSAP 05)}, 129--132, 2005.

\bibitem{Nikiforov}
V.~Nikiforov,
A contribution to the Zarankiewicz problem,
\emph{Linear Algebra Appl.}, 432: 1405--1411, 2010.

\bibitem{PippengerSpencer}
N.~Pippenger, J.~Spencer,
Asymptotic behavior of the chromatic index for hypergraphs,
\emph{J. Combin. Theory Ser. A}, 51: 24--42, 1989.

\bibitem{Qi05}
L.~Qi,
Eigenvalues of a real supersymmetric tensor,
\emph{J. Symb. Comput.}, 40: 1302--1324, 2005.

\bibitem{SFK}
C.-M.~She, Y.-Z.~Fan, L.~Kang,
Spectral bipartite Tur\'an problems on linear hypergraphs,
\emph{Discrete Math.}, 348(6): 114435, 2025.

\bibitem{SzaboSpectrum}
T.~Szab\'o,
On the spectrum of projective norm-graphs,
\emph{Inform. Process. Lett.}, 86(2): 71--74, 2003.

\end{thebibliography}
\end{document}